\documentclass[11 pt]{amsart}
\usepackage{amsmath}
\usepackage{amsxtra}
\usepackage{amscd}
\usepackage{amsthm}
\usepackage{amsfonts}
\usepackage{amssymb}
\usepackage{eucal}
\usepackage{graphicx}
\usepackage{wrapfig}
\usepackage{hyperref}
\usepackage{enumerate}
\graphicspath{ {images/} }

\usepackage[numbers]{natbib}

\usepackage[lmargin=3.2 cm, rmargin=3.2 cm, reversemp]{geometry}

\usepackage{caption}
\usepackage{xcolor}
\usepackage{diagrams}
\definecolor{mygrey}{gray}{0.45}
\definecolor{blue}{RGB}{63, 105, 225}
\definecolor{pink}{RGB}{180, 105, 90}
\definecolor{red}{RGB}{180,0,0}
\definecolor{green}{RGB}{60, 120, 60}

\newtheorem{thm}{Theorem}[section]
\newtheorem{cor}[thm]{Corollary}
\newtheorem{lemma}[thm]{Lemma}
\newtheorem{propn}[thm]{Proposition}

\theoremstyle{definition}
\newtheorem{defn}[thm]{Definition}

\newtheorem*{claim}{Claim}

\newtheorem*{question}{Question}

\theoremstyle{remark}

\newcommand\nc{\newcommand}
\nc\Span{\text{\rm Span}}
\nc\Id{\text{Id}}
\nc\coker{\text{coker}}

\nc \cc {\mathbb{C}}
\nc \dd {\mathbb{D}}
\nc \ff {\mathbb{F}}
\nc \hh {\mathbb{H}}
\nc \ii {\mathbb{I}}
\nc \jj {\mathbb{J}}
\nc \kk {\mathbb{K}}
\nc \LL {\mathbb{L}} 
\nc \mm {\mathbb{M}}
\nc \nn {\mathbb{N}}
\nc \oo {\mathbb{O}}
\nc \pp {\mathbb{P}}
\nc \qq {\mathbb{Q}}
\nc \rr {\mathbb{R}}
\renewcommand \tt {\mathbb{T}}
\nc \zz {\mathbb{Z}}

\nc\cA{\mathcal{A}}
\nc\cB{\mathcal{B}}
\nc\cC{\mathcal{C}}
\nc\cD{\mathcal{D}}
\nc\cE{\mathcal{E}}
\nc\cF{\mathcal{F}}
\nc\cG{\mathcal{G}}
\nc\cH{\mathcal{H}}
\nc\cK{\mathcal{K}}
\nc\cL{\mathcal{L}}
\nc\cM{\mathcal{M}}
\nc\cN{\mathcal{N}}
\nc\cO{\mathcal{O}}
\nc\cP{\mathcal{P}}
\nc\cQ{\mathcal{Q}}
\nc\cR{\mathcal{R}}
\nc\cS{\mathcal{S}}
\nc\csf{\mathcal{S}\mathcal{F}}

\nc\fg{\mathfrak{g}}
\nc\fm{\mathfrak{m}}
\nc\fp{\mathfrak{p}}
\nc\fso{\mathfrak{so}}
\nc\fsu{\mathfrak{su}}

\nc\Tr{\text{Tr}}
\nc\into{\hookrightarrow}

\nc\st{\text{ s.t. }}
\nc\intense[1]{\textcolor[rgb]{1.00,0.00,0.00}{\textbf{#1}}}

\renewcommand{\(}{\left(}
\renewcommand{\)}{\right)}

\nc\Mat{\text{\rm Mat}}
\nc\GL{\text{\rm GL}}
\nc\SU{\text{\rm SU}}
\nc\SO{\text{\rm SO}}
\nc\SL{\text{\rm SL}}
\nc\Sp{\text{\rm Sp}}
\nc\EL{\text{\rm EL}}

\nc\GEM{\text{\rm GEM}}
\nc\Alt{\text{\rm Alt}}
\nc\Sym{\text{\rm Sym}}
\nc\Hess{\text{Hess}}

\nc\Crit{\text{Crit}}

\renewcommand{\(}{\left(}
\renewcommand{\)}{\right)}

\nc\inject{\hookrightarrow}

\nc{\mattwo}[4]{\left[\begin{array}{cc} #1  & #2\\  #3 & #4 \\ \end{array} \right]}
\nc{\matthree}[9]{\left[\begin{array}{ccc} #1  & #2 & #3\\  #4 & #5 & #6 \\ #7 & #8 & #9 \\ \end{array} \right]}
\nc{\vecttwo}[2]{\left[\begin{array}{c} #1 \\ #2 \\ \end{array} \right]}
\nc{\vectthree}[3]{\left[\begin{array}{c} #1 \\ #2\\  #3 \\ \end{array} \right]}

\nc{\del}{\partial}
\nc\onto{\twoheadrightarrow}

\nc\const{\text{const}}

\nc\rrp{\rr P}
\nc\ul{\underline}
\nc\ol{\overline}
\nc\uline{\underline}
\nc\oline{\overline}
\nc\oset{\overset}
\nc\uset{\underset}

\nc\heart{\heartsuit}
\nc\spade{\spadesuit}
\nc\club{\clubsuite}

\nc\marg[1]{\marginnote{\boxed{\text{#1}}}}
\nc\margq[1]{\marginnote{\textcolor[rgb]{1.00,0.00,0.00}{#1}}}

\nc\Char{\text{Char}}
\nc\Frac{\text{Frac}}
\nc\wo{\backslash}
\nc\diag{\text{diag}}
\nc\wtl{\widetilde}
\nc\nsubgp{\triangleleft}

\nc\Cay{\text{Cay}}
\nc\Hom{\text{Hom}}
\nc\Gp{\text{Gp}}
\nc\Set{\text{Set}}
\nc\la{\langle}
\nc\ra{\rangle}

\nc\Spec{\text{Spec}}
\nc\ad{\text{ad}}

\nc\wht{\widehat}
\nc\ddx[2]{\frac{\partial {#1}}{\partial {#2}}}
\nc\dddx[3]{\frac{\partial^2 {#1}}{\partial {#2}\partial{#3}}}
\nc\mult{\text{mult}}
\nc\supp{\text{supp}}
\nc\sign{\text{sign}}

\nc\tr{\text{tr}}
\nc\stab{\text{stab}}
\nc\im{\text{im}}

\nc\sech{\text{sech}}

\nc\Ind{\text{Ind}}
\nc\tsf{\text{sf}}
\nc\End{\text{End}}
\nc\Hol{\text{Hol}}
\nc\hol{\text{hol}}
\nc\Lie{\text{Lie}}

\nc\vol{\text{vol}}
\nc\ind{\text{ind}}
\nc\Met{\mathcal{M}\text{et}}
\nc\Grass{\mathcal{G}\text{rass}}

\nc\longto{\longrightarrow}
\nc\grad{\text{grad}}
\nc\Map{\text{Map}}
\nc\Indx{\text{Ind}}
\nc\indx{\text{ind}}
\nc\pt{\text{pt}}
\nc\Aut{\text{Aut}}
\nc\Br{\text{Br}}
\nc\Gr{\text{Gr}}
\nc\CS{\text{CS}}

\nc\Proj{\text{Proj}}
\nc\Ad{\text{Ad}}
\nc\Diff{\text{Diff}}
\nc\sing{\text{sing}}
\nc\order{\text{order}}
\nc\bsl{\backslash}
\nc\dom{\text{dom}}

\nc\Tor{\text{Tor}}

\title{On the Kronheimer-Mrowka concordance invariant}
\author{Sherry Gong}
\date{}

\begin{document}
\maketitle

\begin{abstract}

Kronheimer and Mrowka introduced a new knot invariant, called $s^\sharp$, which is a gauge theoretic analogue of Rasmussen's $s$ invariant. In this article, we compute Kronheimer and Mrowka's invariant for some classes of knots, including algebraic knots and the connected sums of quasi-positive knots with non-trivial right handed torus knots. These computations reveal some unexpected phenomena: we show that $s^\sharp$ does not have to agree with $s$, and that $s^\sharp$ is not additive under connected sums of knots.

Inspired by our computations, we separate the invariant $s^\sharp$ into two new invariants for a knot $K$, $s^\sharp_+(K)$ and $s^\sharp_-(K)$, whose sum is $s^\sharp(K)$. We show that their difference satisfies $0 \leq s^\sharp_+(K) - s^\sharp_-(K) \leq 2$. This difference may be of independent interest.

We also construct two link concordance invariants that generalize $s^\sharp_\pm$, one of which we continue to call $s^\sharp_\pm$, and the other of which we call $s^\sharp_I$. To construct these generalizations, we give a new characterization of $s^\sharp$ using immersed cobordisms rather than embedded cobordisms. We prove some inequalities relating the genus of a cobordism between two links and the invariant $s^\sharp$ of the links. Finally, we compute $s^\sharp_\pm$ and $s^\sharp_I$ for torus links.

\end{abstract}

\section{Introduction}

In \cite{gtri}, Kronheimer and Mrowka introduced a knot invariant, $s^\sharp$, based on a variant of singular instanton homology with local systems. It is analogous to Rasmussen's invariant, $s$, as defined in \cite{rasmussen}, and it shares some properties with $s$, such as providing a lower bound for the smooth slice genus. More generally, in \cite{gtri}, Kronheimer and Mrowka showed that $s^\sharp$ provides a lower bound for the genus of a surface in any negative definite 4-manifold with $b_1=0$.  Like $s$, the invariant $s^\sharp$ defines a map from the group of smooth concordance classes of knots to $\zz$, though we will see in this paper that, unlike the map that $s$ induces, the map that $s^\sharp$ induces is not a homomorphism.

In this paper we introduce two natural generalizations of $s^\sharp$ to links, one of which is analogous to the $s$ invariants for links defined by Pardon in \cite{pardon}; this version gives us $2^l$ invariants for a link $L$ with $l$ components. Similarly to how $s^\sharp$ was defined for knots in \cite{gtri}, $s^\sharp(L)$ will count the number of factors picked up by the map induced by a cobordism $\Sigma$ from the unlink to $L$ on a version of instanton homology, $I'(L) = I^\sharp(L;\Gamma)/\text{tors}$, which is seen as a module over the ring $\qq[[\lambda]]$. For a cobordism $\Sigma$ from $L_1$ to $L_2$, we write $\psi^\sharp(\Sigma):I'(L_1) \to I'(L_2)$ for the induced map on $I'$.

Let $U_l$ denote the unlink with $l$ components. Then $I'(U_l)$ has $2^l$ generators, which we denote $u_I$ for $I \in \{\pm 1\}^l$. 

\begin{defn} Let $L$ be a link with $l$ components.  Let $\Sigma:U_l \to L$ be a component preserving immersed cobordism that has even genus on each component. Then we define the invariant $s^\sharp_I(L)$ for $I \in \{\pm 1\}^l$ to be
\[s^\sharp_I(L) = g(\Sigma) + p(\Sigma) - m_I(\Sigma)\]
where $g$ is the genus, $p$ is the number of positive double points, and $m_I(\Sigma)$ is the largest integer for which we may write $\psi^\sharp(\Sigma)(u_I) = \lambda^{m_I(\Sigma)}v_I$ for an element $v_I \in I'(L)$.
\end{defn}

We will show that $s^\sharp_I(L)$ is independent of the choice of $\Sigma$ in this definition, and we will also see how to compute it in terms of component-preserving cobordisms that do not have to have even genus on each component; it will be essentially the same, but involve more book-keeping.

We will show that these invariants satisfy an inequality for links related by component-preserving cobordisms. 
\begin{thm} Let $\Sigma:L_1 \to L_2$ be a component-preserving cobordism between links with $l$ components. Then 
\[s^\sharp_I(L_2)-s^\sharp_{(-1)^g(I)}(L_1) \leq \sum_k 2 \(\left\lfloor \frac{g_k(\Sigma )}{2}\right\rfloor + c_k(\Sigma,I)\) + p(\Sigma).\]
$g_k(\Sigma)$ is the genus of the $k$th component, $c_k$ is a book-keeping term that is either $0$ or $1$, $p(\Sigma)$ is the number of positive double points of $\Sigma$, and $(-1)^g$ is a map on indices involving the parity of the genera of the components of $\Sigma$.
\label{thm_s_sharp_I_inequality}
\end{thm}

The other $s^\sharp$ invariants for links that we construct are analogous to the link $s$ invariant defined by Beliakova and Wehrli in \cite{beliakova_wehrli}. They are defined as follows.

\begin{defn} Let $U$ denote the unknot. For a link $L$, let $\Sigma:U \to L$ be an immersed cobordism. The group $I'(U)$ is a free module of rank two over $\qq[[\lambda]]$. Let $u_+$ and $u_-$ denote its generators. Let $g(\Sigma)$ denote the genus of $\Sigma$, $p(\Sigma)$ denote its number of positive double points, and $m_\pm(\Sigma)$ denote the largest integer for which we may write $\psi^\sharp(\Sigma)(u_\pm) = \lambda^{m_\pm(\Sigma)}v$ for an element $v \in I'(L)$.

Then we define the invariant $s^\sharp_\pm(L)$
\[s^\sharp_\pm(L) =\begin{cases} g(\Sigma) + p(\Sigma) - m_\pm(\Sigma) & \text{if $\Sigma$ has even genus} \\ g(\Sigma) + p(\Sigma) - m_\mp(\Sigma)\pm 1  & \text{else.} \end{cases}\]
\end{defn}

We will also combine these two invariants into a single invariant $s^\sharp(L) = s^\sharp_+(L)+ s^\sharp_-(L)$, which will be analogous to Kronheimer and Mrowka's $s^\sharp$ for knots. We will show the following inequality.
\begin{thm} Let $\Sigma:L_1 \to L_2$ be an immersed cobordism such that every component of $\Sigma$ has non-trivial boundary in $L_1$. Then 
\[s^\sharp(L_2) - s^\sharp(L_1) \leq 2p(\Sigma) - \chi(\Sigma) +\ell, \]
where $\ell = \#(L_1) -\#(L_2)$ where $\#(L_i)$ is the number of components of $L_i$. If $\Sigma$ is embedded, then $p(\Sigma) = 0$. 
\label{thm_s_sharp_cobordism_ineq}
\end{thm}

Finally, we compute the invariant $s^\sharp$ for algebraic knots, connected sums of quasi-positive knots with torus knots, and torus links. We will do this primarily using results about complex curves. We will say that a complex curve is irreducible in a ball if its intersection with the ball has only one irreducible component. This condition is neither stronger nor weaker than the condition that the curve be irreducible.  We will show the following lemma.
\begin{lemma} If there is complex curve $\Sigma \subset B^4 \subset \cc^2$, which is irreducible in $B^4$ with boundary $L \subset S^3$ that is embedded away from one transverse double point, then $s^\sharp_+(L)= g$ and $s^\sharp_-(L) = g-1$, where $g$ is the genus of an embedded surface that is a perturbation of $\Sigma$ as a complex curve (eg. by replacing the curve $P(x,y) = 0$ with $P(x,y) +\epsilon = 0$ for a small number $\epsilon$).
\label{lemma_curve_with_double_pt}
\end{lemma}
The genus of a smooth embedded complex curve is also the minimal genus of a smooth embedded connected surface with boundary $L$, as can be shown using $s^\sharp$ or by other methods. 

We will then construct such irreducible complex curves for the following links:
\begin{itemize}
\item Non-trivial algebraic knots, including non-trivial right-handed torus knots
\item Right-handed torus links $T_{m,n}$ other than $T_{2,2}$
\item Knots which are connected sums $K \# T_{p,q}$ where $K$ is a quasi-positive knot, and $T_{p,q}$ is a non-trivial right-handed torus knot.
\end{itemize}

As Kronheimer and Mrowka did for the $s^\sharp$ invariant of knots, observe that we have added $s^\sharp_+$ and $s^\sharp_-$ to get $s^\sharp$. However, it may also be of interest to consider $s^\sharp_+(L)-s^\sharp_-(L)$. We will show the following.

\begin{propn} For any link $L$, 
\[0 \leq s^\sharp_+(L) -s^\sharp_-(L) \leq 2.\]
\label{propn_s_plus_minus_difference}
\end{propn}
For the unknot, we have $s^\sharp_+(L)-s^\sharp_-(L) = 0$, and for any links that fit into the paradigm of Lemma \ref{lemma_curve_with_double_pt}, we have $s^\sharp_+(L)-s^\sharp_-(L) = 1$. 

We can think of $s^\sharp_+(L) - s^\sharp_-(L) \in \{0,1,2\}$ as a link invariant. In the case of a knot, this link invariant can also be characterised as follows. 

\begin{propn} Let $C$ denote the cylinder from $K \times I \subset S^3 \times I$. Let $C'$ denote the connected sum of $C$ with a standard embedded torus, so that $C'$ has genus $1$. Then, the map that $C'$ induces on $I'$ is given by
\[ \psi^\sharp(C') \sim \begin{cases} 
\mattwo{0}{\lambda^2}{1}{0} & \text{if } s^\sharp_+(L) -s^\sharp_-(L) = 0\\
 \mattwo{0}{\lambda}{\lambda}{0} & \text{if } s^\sharp_+(L) -s^\sharp_-(L) = 1\\
 \mattwo{0}{1}{\lambda^2}{0}  & \text{if } s^\sharp_+(L) -s^\sharp_-(L) = 2
\end{cases}\]
in the $u_+,u_-$ basis, where $\sim$ means that the entries are up to scaling by non-multiples of $\lambda$.
\label{propn:difference_as_matrix}
\end{propn}

We will use similar arguments to prove the following inequalities for quasi-positive links.

\begin{propn} Let $g$ be the 4-ball genus of a quasi-positive knot $K$. Then if $g$ is even, $s^\sharp_+(K) = g$ and $g-1 \leq s^\sharp_-(K)  \leq g$. If $g$ is odd, then $g \leq s^\sharp_+(K) \leq g+1$ and  $s^\sharp_-(K) = g-1$.
\label{propn_quasi-positive_inequality}
\end{propn}

We have the following computation for $s^\sharp_I$.

\begin{thm} Let $m,n>1$ be relatively prime positive integers. Let $T_{md,nd}$ be the $(md,nd)$ torus link, which has $d$ components. Then 

\[s^\sharp_I(T_{md,nd}) =  \begin{cases} d\frac{(m-1)(n-1)}{2}+mn\frac{d(d-1)}{2} - n(I) & \text{if $\frac{(m-1)(n-1)}{2}$ is even}\\
d\frac{(m-1)(n-1)}{2}+d+mn\frac{d(d-1)}{2} - 3n(I) & \text{else.}\end{cases}\]
where $n(I)$ is the number of $(-1)$s in $I$.
\label{thm_s_sharp_I_links}
\end{thm}

Finally, we will show the following result about knots that can be obtained from the unknot by switching crossings.

\begin{thm} Let $D_K$ be a knot diagram for the knot $K$. Suppose that there is a subset $S_+$ of the positive crossings of $D_K$, such that switching which strand is on top for those crossings results in an unknot. Suppose further that there is a subset $S_-$ of the negative crossings such that switching which strand is on top for those crossings also yields the unknot. Then $s^\sharp_\pm(K) = 0$.
\label{thm_switching_crossings}
\end{thm}

Kronheimer and Mrowka showed in \cite{gtri} that, for any knot $K$ we have $s^\sharp(m(K)) = -s^\sharp(K)$, where $m(K)$ is the mirror of $K$. Hence, $s^\sharp(K)=0$ for amphichiral knots and slice knots. Theorem \ref{thm_switching_crossings} applies to some amphichiral knots such as the figure-eight, but it also applies to some non-amphichiral, non-slice knots such as $7_7$ and $8_1$.

\section*{Acknowledgements} I would like to thank Ciprian Manolescu for his interest in this project and for many inspired conversations. I would like to thank Peter Kronheimer and Tomasz Mrowka for many helpful mathematical conversations and email communications. I would also like to thank Peter Feller for helpful discussions about complex curves and the knots that bound them.

\section{Background}

In \cite{gtri}, Kronheimer and Mrowka defined a gauge theoretic version of the Rasmussen $s$ invariant using a version of instanton homology with local systems. Let us begin with a review of their construction.

They started by constructing an instanton homology with local coefficients for a link $L$, which they denoted $I^\sharp(L,\Gamma)$. It comes from an $S^1$ valued map on the configuration space, given by taking a connection $[A]$ to the product of the holonomies of $A$ along components of the link. This $I^\sharp(L,\Gamma)$ is a module over $\qq[u, u^{-1}]$, where $u$ keeps track of the holonomy.

An immersed cobordism $\Sigma$ between links $L_0$ and $L_1$ induces a map 
\[\psi^\sharp(\Sigma):I^\sharp(L_0) \to I^\sharp(L_1).\]

If two immersed cobordisms $\Sigma$ and $\Sigma'$ are homotopic (not necessarily through immersions), then one of the maps induced by $\Sigma$ and $\Sigma'$ is a constant multiple of the other, and this constant multiple is a power of $(1-u^2)$. 

Kronheimer and Mrowka explicitly describe how the maps differ. Note that $\Sigma'$ can be obtained from $\Sigma$ by a sequence of local moves and their inverses. These local moves are of three types: positive twist moves, which introduce a positive double point, negative twist moves, which introduce a negative double point, and finger moves, which introduce a cancelling pair of positive and negative double points. 

By Proposition 3.1 of \cite{gtri}, these local moves change the maps $\psi^\sharp(\Sigma)$ as follows: for $\Sigma'$ obtained from $\Sigma$ by a local move,
\begin{itemize}
\item For a positive twist move, $\psi^\sharp(\Sigma') = (1-u^2) \psi^\sharp(\Sigma)$. 
\item For a negative twist move, $\psi^\sharp(\Sigma') =  \psi^\sharp(\Sigma)$.  
\item For a finger move, $\psi^\sharp(\Sigma') = (1-u^2) \psi^\sharp(\Sigma)$. 
\end{itemize}

In particular, if $\Sigma$ and $\Sigma'$ are homotopic and have the same number of positive double points, then they induce the same map.

Another consequence of the above is that the torsion part of $I'(L)$ is killed by some factor of $(1-u)^2$, and the torsion free part, $I'(L) = I^\sharp(L;\Gamma)/\text{tors}$, is a free module of rank $2^l$ where $l$ is the number of components of $L$.

In the case of a knot, $K$, $I^\sharp(K; \Gamma)$ comes with a mod 4 grading and $I'(K)$ has one generator in each of degrees $1$ and $-1$ mod $4$. For the unknot, $U$, $I^\sharp(U;\Gamma)$ is a free module of rank $2$, and we denote the generators in degrees $1$ and $-1$ by $u_+$ and $u_-$, respectively.

For the embedded torus, $T$, as a cobordism from the unknot to itself, the map induced is
\[\psi^\sharp(T)(u_+) = 2 u_-\]
\[\psi^\sharp(T)(u_-) = 2u^{-1}(u-1)^2 u_-\]

Let $\cR$ denote the completion of $\qq[u,u^{-1}]$ at $u=1$, which we write as $\qq[[\lambda]]$ where $\lambda^2 = u^{-1}(u-1)^2$.

For a knot $K$, let $v_+$ and $v_-$ denote the generators of $I'(K)$ in degrees $+1$ and $-1$ modulo 4. For a cobordism $\Sigma$ from the unknot to $K$, $\psi^\sharp(u_+)$ and $\psi^\sharp(u_-)$ are multiples of $v_+$ and $v_-$, not necessarily respectively. Indeed if the genus of $\Sigma$ is even, then we have
\begin{align*}
\psi^\sharp(\Sigma)(u_+) & = \sigma_+(\Sigma) v_+\\
\psi^\sharp(\Sigma)(u_-) & = \sigma_-(\Sigma) v_-\\
\end{align*}
and, if the genus is odd,
\begin{align*}
\psi^\sharp(\Sigma)(u_+) & = \sigma_+(\Sigma) v_-\\
\psi^\sharp(\Sigma)(u_-) & = \sigma_-(\Sigma) v_+.\\
\end{align*}

Kronheimer and Mrowka define two quantities $m_+(\Sigma)$ and $m_-(\Sigma)$, to denote the number of factors of $\lambda$ in $\sigma_+$ and $\sigma_-$, respectively, as elements of $\qq[[\lambda]]$. 

Then, they defined $s^\sharp$ as
\[s^\sharp(K) = 2g(\Sigma)-(m_+(\Sigma)+m_-(\Sigma)).\]

\section{Constructing the link invariants}

\subsection{Getting $2^l$ link invariants}

In this subsection, we define a link analogue for $s^\sharp$, which assigns an integer for each of the $2^l$ standard generators of $I'(U_l)$, where $l$ is the number of components of $L$ and $U_l$ is the unlink with $l$ components. Like the invariant $s^\sharp$ for knots, it will count factors of $\lambda$ in the image of the generators under maps induced by cobordisms.

Let us start with the case of an immersed cobordism $\Sigma$ from the unlink to $L$. It induces a map
\[\psi^\sharp(\Sigma): I'(U_l) \to I'(L).\]

Observe that for a disjoint union of links $L_1 \coprod L_2$, $I'(L_1 \coprod L_2) = I'(L_1) \otimes I'(L_2)$, by the same excision argument as found in section 5 of \cite{km_unknot}. More specifically, by that excision argument it can be shown that 
\[C^\sharp(L_1,\Gamma) \otimes C^\sharp(L_2,\Gamma) \simeq C^\sharp(L_1 \amalg L_2,\Gamma) \otimes C^\sharp(\emptyset,\Gamma).\]
Then, because $C^\sharp(\emptyset,\Gamma) = \qq[[\lambda]]$, we get
\[C^\sharp(L_1,\Gamma) \otimes C^\sharp(L_2,\Gamma) \simeq C^\sharp(L_1 \amalg L_2,\Gamma).\]
This gives us a split exact sequence
\[0 \to I^\sharp(K_1;\Gamma) \otimes I^\sharp(K_2;\Gamma) \to I^\sharp(K_1 \amalg K_2;\Gamma) \to \Tor_1(I^\sharp(K_1;\Gamma), I^\sharp(K_2;\Gamma)) \to 0.\]
The fact that the sequence is split shows that on the torsion free parts, we have
\[I'(K_1) \otimes I'(K_2)  \simeq I'(K_1 \amalg K_2).\]

Moreover, this isomorphism is functorial with respect to cobordisms $L_1 \to L_1'$ and $L_2 \to L_2'$, because the excision map is induced by a cobordism of three manifolds with knots, and it commutes with the maps induced by cobordism on the knots.

Thus, $I'(U_l) = I'(U)^{\otimes l}$.

Fix a set of generators of $I'(U_l)$, $u_I$ for $I \in \{(i_1, i_2,\cdots , i_l)| i_j \in \{-1,1\}\}$, given by 
\[u_I = u_{i_1} \otimes u_{i_2} \otimes \cdots \otimes u_{i_n} \]
where $u_{-1} = u_-$ and $u_1 = u_+$ are the generators of $I^\sharp(U)$ defined in the introduction.

For $I =(i_1,i_2,\ldots , i_l)$ with $i_j \in \{0,1\}$, and $(g_1,g_2,\ldots g_l)$ a tuple of integers, let 
\[(-1)^g(I) = ((-1)^{g_1}i_1, (-1)^{g_2}i_2,\ldots, (-1)^{g_l}i_l),\]

\begin{defn} For an immersed cobordism, $\Sigma$, we define the invariant
\[s^\sharp_I(L) = \sum_k 2 \(\left\lfloor \frac{g_k(\Sigma)}{2}\right\rfloor+c_k(\Sigma,I)\)+p -m_{(-1)^{g(\Sigma)}(I)}(\Sigma)\]
where $p$ is the number of positive double points, $g(\Sigma) = (g_1,\ldots g_l)$, with $g_i$ being the genus corresponding to the $k$th component,
 \[c_k = \begin{cases} 1 & \text{if $i_k = 1$ and $g_k$ is odd} \\ 0 & \text{if $i_k = -1$ or $g_k$ is even}\end{cases}.\]
It is easy to see that this recovers the definition in the introduction when $\Sigma$ has even genus on each component.
\label{s_sharp_I_defn}
\end{defn}

\begin{propn} The invariant $s^\sharp_I(L)$ is well defined, i.e., it does not depend on the choice of immersed cobordism $\Sigma$.
\end{propn}

\begin{proof}  
Note that for any $l$-tuple of non-negative integers, there is an immersed cobordism with those values as its genus on the $l$ components; one can see this because it is possible to turn the link into the unlink by pulling arcs through each other, which yields an immersed cobordism that is composed of annuli. It is then easy to add genus to each component as desired.

Moreover, observe that the values $s^\sharp$ for two component preserving cobordisms from $U_l$ to $L$ with the same genus on each component are the same. This is because the cobordisms are homotopic through a linear interpolation between the immersions; the homotopy can then be broken down into positive and negative twist moves and finger moves, which each change the number of positive double points, $p(\Sigma)$, and the number of factors of $\lambda$ picked up by $u_{(-1)^{g(\Sigma)}(I)}$ by the same amount. 

We will first show the statement for the case where $\Sigma$ has even genus on each cobordism, and then we will show that if $\Sigma_1$ and $\Sigma_2$ have the same genus on each component except one, and $\Sigma_1$ has even genus on that component, and $\Sigma_2$ has one higher genus, then they define the same $s^\sharp_I$ for each $I$.

For the case where $\Sigma_1$ and $\Sigma_2$ are two immersed, component-preserving cobordisms with each component having even genus, then there are component preserving, orientable, embedded surfaces $T_1$ and $T_2$, of even genus on each component, from the unlink to itself, such that $\Sigma_1 \circ T_1$ and $\Sigma_2 \circ T_2$ have the same genus on each component. Then $\Sigma_1 \circ T_1$ and $\Sigma_2 \circ T_2$ are homotopic, for example, by linearly interpolating between them. Thus, 
\[\lambda^{p(\Sigma_2)}\psi^\sharp(\Sigma_1 \circ T_1) \sim \lambda^{p(\Sigma_1)}\psi^\sharp(\Sigma_2 \circ T_2)\] 
where $p(\Sigma)$ is the number of positive double points of $\Sigma$, and $\sim$ means that the two sides are the same up to a constant multiple, where the multiple is invertible in $\qq[[\lambda]]$, so it does not affect the number of factors of $\lambda$. 

Observe also that $\psi^\sharp(T_1) \sim \lambda^{g(T_1)}$ and $\psi^\sharp(T_2) \sim\lambda^{g(T_2)}$. This is where we are using that $T_1$ and $T_2$ have even genus on each component, and it was computed in \cite{gtri} that adding genus two to one component causes the map to be changed by a constant multiple which is $\sim \lambda^2$. 

Thus,
\[\lambda^{p(\Sigma_2) + g(T_1)} \psi^\sharp(\Sigma_1)  \sim \lambda^{p(\Sigma_1) + g(T_2)} \psi^\sharp(\Sigma_2),\]
so
\[m_I(\Sigma_1) + p(\Sigma_2) + g(T_1) = m_I(\Sigma_2) + p(\Sigma_1) + g(T_2).\]
Note also that $g(T_1) + g(\Sigma_1) = g(T_2) + g(\Sigma_2)$, so the above means
\[m_I(\Sigma_1) - p(\Sigma_1) - g(\Sigma_1) = m_I(\Sigma_2) - p(\Sigma_2) - g(\Sigma_2),\]
as desired.

Finally consider the case $\Sigma_1 = \Sigma \circ T$, where $T$ is a cobordism from the unlink to itself that has genus $1$ on one component and is otherwise cylindrical. As computed in \cite{gtri}, the map induced by $T$ is
\[\psi^\sharp(T)(v_+) = 2v_- \text{, } \psi^\sharp(T)(v_-) = 2 \lambda^2 v_+,\]
where we are only keeping track of the term corresponding to the component that where the genus is, as computed in \cite{gtri}. It is now easy to see that $s^\sharp_I(L)$ computed from $\Sigma_1$ and $\Sigma_2$ yield the same value.
\end{proof}

\subsection{Two invariants, $s^\sharp_+$ and $s^\sharp_-$, and a single invariant, $s^\sharp$}

We can also use a similar set-up to obtain an invariant that respects more general cobordisms: For a connected cobordism $\Sigma$ from the unknot to $L$, let $m_+(\Sigma)$ and $m_-(\Sigma)$ be the number of factors of $\lambda$ picked up by $u_+$ and $u_-$, respectively.

\begin{defn}
For a cobordism $\Sigma$ from the unknot to the link, we define 
\[s^\sharp_\pm (L) = g(\Sigma)+p(\Sigma)-m_\pm(\Sigma).\]
when $\Sigma$ has even genus. When $\Sigma$ has odd genus, we define
\[s^\sharp_{\pm}(L) = g(\Sigma)+p(\Sigma) -m_\mp(\Sigma) \pm 1.\]
We may also combine these into a single invariant, which agrees with Kronheimer and Mrowka's $s^\sharp$ for knots, by defining
\[s^\sharp(L) = 2g(\Sigma)+2p(\Sigma)-m_+(\Sigma)-m_-(\Sigma).\]
\end{defn}

This is an invariant for the same reason that the $2^l$ invariants defined above were: for two connected cobordisms $\Sigma_1$ and $\Sigma_2$ with $g(\Sigma_1) \geq g(\Sigma_2)$, $\Sigma_1$ can be homotoped to $\Sigma_2 \circ T$, where $T$ is an embedded oriented cobordism from the unknot to itself. Then, one can easily check that the twist moves, finger moves, and addition of handles to $T$ (which is embedded and goes from the unknot to itself) preserve $s^\sharp$.

Unlike in the case of the Rasmussen invariant, where the two quantities defined by different generators are always related by a constant, so that there is no information lost by combining them into one, $s^\sharp_+(L)- s^\sharp_-(L)$ may be different for different links; indeed, we will see that the value for the unknot is different from that of the trefoil. However, there are bounds for $s^\sharp_+(L)- s^\sharp_-(L)$, stated in Proposition \ref{propn_s_plus_minus_difference}, which we restate and prove here.

\begin{propn} For any link $L$, $0 \leq s^\sharp_+(L)-s^\sharp_-(L)\leq 2$. 
\end{propn}
\begin{proof}  As we mentioned earlier, $I'(L_1 \coprod L_2) \simeq I'(L_1) \otimes I'(L_2)$ in a way that is natural with respect to cobordisms $L_1 \to L_1'$ and $L_2 \to L_2'$. 

For any link $L$ let $\mu_{L}:U \coprod L \to L$ denote any one of the maps induced by the cobordism merging the unknot into one component of $L$, and let $\psi^\sharp(\mu_L): I'(U) \otimes I'(L) \to I'(L)$ denote the induced map on $I'$. This map will depend on which component of $L$ is chosen, but, for the purposes of computation of $s^\sharp_\pm$, as long as for each link $L$, we fix one of the components and only use $\mu_L$ for that component, this will not matter.

Consider a cobordism of even genus $\Sigma$ from the unknot to $L$. Then the composite cobordism $\mu_{L} \circ \Id \coprod \Sigma$ from $U \coprod U$ to $L$ is homotopic to the cobordism $\Sigma \circ \mu_U$, and they have the same number of positive double points, so they induce the same map on $I'$. Thus we have
\begin{align*}
\lambda^{m_+(\Sigma)}  \psi^\sharp(\mu_L)  (u_- \otimes v_+) &= \psi^\sharp(\mu_L)  (u_- \otimes \lambda^{m_+(\Sigma)}v_+)\\
& = \psi^\sharp(\mu_L) \circ (\Id \otimes \psi^\sharp(\Sigma)) (u_- \otimes u_+)\\
& = \psi^\sharp(\Sigma) \circ \psi^\sharp(\mu_U)(u_- \otimes u_+) \\
& = \psi^\sharp(\Sigma)( u_-) = \lambda^{m_-(\Sigma)}v_-,\\
\end{align*}
where $\lambda$ is not a factor of $v_-$, so $m_-(\Sigma) \geq m_+(\Sigma)$. Thus, $s^\sharp_-(L) \leq s^\sharp_+(L)$, because $\Sigma$ was chosen to have even genus. 

Doing the same computation for $u_-$, we see that 
\begin{align*}
\lambda^{m_-(\Sigma)}  \psi^\sharp(\mu_L)  (u_- \otimes v_-) &= \psi^\sharp(\mu_L)  (u_- \otimes \lambda^{m_-(\Sigma)}v_-)\\
& = \psi^\sharp(\mu_L) \circ (\Id \otimes \psi^\sharp(\Sigma)) (u_- \otimes u_-)\\
& = \psi^\sharp(\Sigma) \circ \psi^\sharp(\mu_U)(u_- \otimes u_-) \\
& = \psi^\sharp(\Sigma) (\lambda^2 u_+) = \lambda^{2+m_+(\Sigma)}v_+,\\
\end{align*}
and, again $\lambda$ is not a factor of $v_+$, so $m_-(\Sigma) \leq 2 + m_+(\Sigma)$, so $s^\sharp_-(L) \geq 2+s^\sharp_+(L)$, again because $\Sigma$ was chosen to have even genus.
\end{proof}

As mentioned in the introduction in Proposition \ref{propn:difference_as_matrix}, in the case of knots, there is another way to characterise $s^\sharp_+(K)-s^\sharp_-(K)$, which we can view as another knot invariant in itself. We now restate and prove this proposition.

\begin{propn} Let $C$ denote the cylinder from $K \times I \subset S^3 \times I$. Let $C'$ denote the connected sum of $C$ with a standard embedded torus, so that $C'$ has genus $1$. Then, the map that $C'$ induces on $I'$ is given by
\[ \psi^\sharp(C') \sim \begin{cases} 
\mattwo{0}{\lambda^2}{1}{0} & \text{if } s^\sharp_+(L) -s^\sharp_-(L) = 0\\
 \mattwo{0}{\lambda}{\lambda}{0} & \text{if } s^\sharp_+(L) -s^\sharp_-(L) = 1\\
 \mattwo{0}{1}{\lambda^2}{0}  & \text{if } s^\sharp_+(L) -s^\sharp_-(L) = 2
\end{cases}\]
in the $u_+,u_-$ basis, where $\sim$ means that the entries are up to scaling by non-multiples of $\lambda$.
\end{propn}
\begin{proof}  and $\Sigma$ be an embedded cobordism of even genus from the unknot to $K$ Then the composite cobordism $C \circ \Sigma'$ and $C' \circ \Sigma$ induce the same map. Suppose that the map that $C'$ induces is
\[\mattwo{0}{b}{a}{0}.\]

In terms of elements, $C' \circ \Sigma:U \to K \to K$ acts by
\[u_+ \mapsto \hspace{0.2 cm} \sim \lambda^{m_+}v_+ \mapsto \hspace{0.2 cm} \sim \lambda^{m_++a}v_-\]
\[u_- \mapsto \hspace{0.2 cm} \sim \lambda^{m_-}v_- \mapsto \hspace{0.2 cm} \sim \lambda^{m_-+b}v_+\]
and $C \circ \Sigma':U \to U \to K$ acts by 
\[u_+ \mapsto \hspace{0.2 cm} \sim u_- \mapsto \hspace{0.2 cm} \sim \lambda^{m_-}v_-\]
\[u_- \mapsto \hspace{0.2 cm} \sim \lambda^2u_+  \mapsto \hspace{0.2 cm} \sim \lambda^{2+m_+}v_+.\]
Thus, $a = m_--m_+$ and $b = 2+m_+-m_-$. The proposition now follows from the definition of $s^\sharp_\pm$. 
\end{proof} 

\section{Cobordism inequalities and concordance invariance}

\subsection{For the $2^l$ invariants} Let us prove Theorem \ref{thm_s_sharp_I_inequality}. 

Let $\Sigma$ be an immersed, oriented, component-preserving cobordism from link $L_1$ to $L_2$. Let $l$ be the number of components of $L_1$. Consider a component preserving cobordism $\Sigma_1$ from the unlink $U_l$ to $L_1$ with even genus on each component.

Then the number of factors of $\lambda$ in $\psi^\sharp(\Sigma \circ \Sigma_1)(u_I)$ is at least the number of factors of $\lambda$ in $\psi^\sharp(\Sigma_1)(u_I)$, so
\[m_I(\Sigma \circ \Sigma_1) \geq m_I (\Sigma_1),\]
where, keeping our notation from the previous section, $(-1)^{g(\Sigma)}(I)$ switches the $u_+$ and $u_-$ for each component according to the genus of $\Sigma$ on that component.
Thus, $m_{(-1)^{g(\Sigma_1)}(I)}(\Sigma \circ \Sigma_1)  \geq m_{(-1)^{g(\Sigma_1)}(I)}(\Sigma_1)$, which we can rewrite as
\[-s^\sharp_I(L_2) + \sum 2 \(\left\lfloor \frac{g_k(\Sigma \circ \Sigma_1)}{2}\right\rfloor + c_k(\Sigma \circ \Sigma_1,I)\) + p(\Sigma \circ \Sigma_1) \geq -s^\sharp_{(-1)^{g(\Sigma)(I)}}(L_1) + g(\Sigma_1) + p(\Sigma_1).\]

Since $\Sigma_1$ is component preserving and has even genus on each component, the above simplifies to 
\[ s^\sharp_I(L_2) -s^\sharp_{(-1)^{g(\Sigma)}(I)}(L_1) \leq \sum 2 \(\left\lfloor \frac{g_k(\Sigma )}{2}\right\rfloor + c_k(\Sigma,I)\) + p(\Sigma),\]
so we have completed the proof of Theorem \ref{thm_s_sharp_I_inequality}. 

We can use Theorem \ref{thm_s_sharp_I_inequality} to deduce the following.

\begin{propn} If there is an embedded cobordism $\Sigma$ between $L_1$ and $L_2$ such that $\Sigma$ is component-wise and each component is topologically an annulus, then $s^\sharp_I(L_1) = s^\sharp_I(L_2)$.
\end{propn}

\begin{proof} Because $\Sigma$ is embedded, we see that $p(\Sigma) = 0$. Because $\Sigma$ is a component-preserving cobordism where each component is an annulus, we see that $s^\sharp_I(L_2) \leq s^\sharp_I(L_1)$. Combining this inequality with the one obtained by running $\Sigma$ backwards from $L_2$ to $L_1$, we see that $s^\sharp_I(L_1) = s^\sharp_I(L_2)$, as desired.
\end{proof}

\subsection{For the invariant $s^\sharp(L)$} We now prove Theorem \ref{thm_s_sharp_cobordism_ineq}.

Let $\Sigma$ be an immersed, oriented cobordism from $L_1$ to $L_2$ such that every component has boundary components in both $L_1$ and $L_2$.

Then for $\Sigma_1$ a connected cobordism from the unknot to $L_1$, $\Sigma \circ \Sigma_1$ is also connected. Because $m_\pm$ counts the number of factors of $\lambda$ picked up, we get
\[m_\pm (\Sigma \circ \Sigma_1) \geq m_\pm (\Sigma_1).\]
The formula for $s^\sharp_\pm$ sees either $m_\pm$ or $m_\mp$, depending on whether the genus is even or odd, but to compute $s^\sharp$ we add $s^\sharp_\pm$ together, so, either way, we get
\[-s^\sharp(L_2) + 2g(\Sigma \circ \Sigma_1) + 2p(\Sigma \circ \Sigma_1) \geq 
-s^\sharp(L_1) +  2g(\Sigma_1) + 2p(\Sigma_1),\]
so
\[-s^\sharp(L_1) +s^\sharp(L_2)  \leq 2p(\Sigma) + 2g(\Sigma \circ \Sigma_1) - 2g(\Sigma_1)\leq 2p(\Sigma) -\chi(\Sigma)+\ell,\]
where $\ell = \#(L_1) -\#(L_2)$ where $\#(L_i)$ is the number of components of $L_i$. If $\Sigma$ is embedded, then $p(\Sigma) = 0$.

\section{Computations of $s^\sharp_\pm$}

\subsection{Links that bound singular complex curves}

We will compute $s^\sharp_\pm$ for certain knots and links using cobordisms that come from complex curves. As in the definition of $s^\sharp_\pm$, we will want to consider only immersions of connected cobordisms. For a complex curve $C$, we will say that $C$ is irreducible in a ball $B \subset \cc^2$ if $C \cap B$ is the image of an immersion of a connected cobordism into $B$.

Another way to phrase this is the following: If $U\subset \Sigma$ is the open subset where the curve is smooth, then we say that $\Sigma$ is irreducible in $B$ if the intersection $B \cap U$ is connected.

Note that a curve that is cut out by an irreducible polynomial does not necessarily satisfy this condition for every ball; for example,  $x^3 + x^2 - y^2$ is irreducible, but in a small ball around $(0,0)$ the surface it cuts out looks like two disks through the origin.

We start with a proof of Lemma \ref{lemma_curve_with_double_pt}, which we restate here.

\begin{lemma} Suppose $L \subset S^3$ bounds an immersed complex curve $\Sigma_1 \subset B^4$, which is embedded except for one transverse double point. We further suppose that $\Sigma_1$ is irreducible in $B$. Then $s^\sharp(L) = 2g  - 1$, where $g$ is the genus of an embedded complex curve with boundary $L$, which necessarily exists. In particular, $s^\sharp_+(L) = g$ and $s^\sharp_-(L) = g-1$.
\label{one_double_point_lemma}
\end{lemma} 

\begin{proof} We may suppose that the singularity is at $(0,0) \in \cc^2$. If the curve is cut out by $P(x,y) = 0$, then $(0,0)$ being a singularity means that there are no constant or linear terms in $P(x,y)$, and the fact that the double point is transverse means that the quadratic term can be written as $(ax+by)(cx+dy)$ where $ad \neq bc$. Applying a linear transformation, we can assume that the quadratic term of $P(x,y)$ is $xy$. Note that the double point had to be a positive double point.

Consider perturbing it to $P(x,y) -\epsilon$ for some small $\epsilon$ to get a surface $\Sigma_2$ in $B^4$ that bounds the same knot (up to isotopy). 

Then $\Sigma_2$ will have one higher genus than $\Sigma_1$, because away from the origin, the addition of $\epsilon$ has no topological effect, and near the origin, the effect of the addition of $\epsilon$ is akin to perturbing the $xy$ as an immersed surface with boundary given by the Hopf link to the Seifert surface of the Hopf link, which is an annulus. The irreducibility in $B$ means that this change adds a tube from one component of the surface to itself. 

Now, we have complex curves $\Sigma_1$ and $\Sigma_2$, both with boundary given by $L$ (up to isotopy), such that $\Sigma_1$ has one transverse double point, and $\Sigma_2$ is embedded, but has one higher genus. Let $(S^3 \times I, \Sigma_1^0)$ and $(S^3 \times I, \Sigma_2^0)$ be the cobordisms in $S^3 \times I$ from $(S^3, U)$ to $(S^3, L)$, where $U$ is the unknot, that arise by puncturing $(B^4, \Sigma_1)$ and $(B^4, \Sigma_2)$ respectively at points on the corresponding surfaces.

According to the non-vanishing result of Section 5.4 of \cite{khgfi}, $\psi^\sharp(\Sigma_1^0)$ and $\psi^\sharp(\Sigma_2^0)$ are both not divisible by $\lambda$, which means that at least one of $m_+(\Sigma_1^0)$ and $m_-(\Sigma_1^0)$ is zero, and the same holds for $\Sigma_2^0$. In \cite{khgfi}, Kronheimer and Mrowka only use this result for knots, but the proof applies to links as well. 

Now consider $T$ to be the genus 1 surface from the unknot to itself (that is the twice punctured torus), and $S$ to be an immersed annulus with one positive double point from the unknot to itself. 

Note that as we have mentioned above, $\psi^\sharp(T)$ induces the map
\begin{align*}
\psi^\sharp(T)(u_+) &=  2u_-\\
\psi^\sharp(T)(u_-) & = 2\lambda^2u_+.
\end{align*}

Because $S$ arises from the cylinder by way of a single positive twist move, $S$ induces the map
\begin{align*}
\psi^\sharp(S)(u_+) &=  c\lambda u_+\\
\psi^\sharp(S)(u_-) & = c \lambda u_-
\end{align*}
for a constant $c$ which is not a multiple of $\lambda$.

Then consider 
\begin{diagram}
U & \rTo^T & U & \rTo^{\Sigma_1^0}  & L\\
\end{diagram}
and
\begin{diagram}
U & \rTo^S & U & \rTo^{\Sigma_2^0} & L\\
\end{diagram}
These two cobordisms have the same genus and the same number of positive double points, so by looking at what the twist and handle moves do, we see that they induce the same map on $I'$.

The first one induces
\[u_+ \mapsto 2 u_- \mapsto 2 \lambda^{m_-(\Sigma_1^0)} v_-\]
\[u_- \mapsto 2\lambda^2 u_+ \mapsto 2 \lambda^{2+ m_+(\Sigma_1^0)} v_+\]

and the second one induces
\[u_+ \mapsto c\lambda v_+ \mapsto c\lambda^{1+m_+(\Sigma_2^0)} v_-\]
\[u_- \mapsto c\lambda v_- \mapsto c\lambda^{1+m_-(\Sigma_2^0)}v_+\]
where $c$ is not a multiple of $\lambda$.

Setting them to be equal, we see that $m_-(\Sigma_2^0) \geq 1$, so by our above discussion, $m_+(\Sigma_2^0) =0$.

Similarly, $m_-(\Sigma_1^0) \geq 1$, so by our above discussion, $m_+(\Sigma_1^0) =0$. Then, comparing the other terms, we see that $m_-(\Sigma_1^0)$ and $m_-(\Sigma_2^0)$ must both be equal to $1$, so, using either $\Sigma_1$ or $\Sigma_2$, we see that $s^\sharp(L) = 2g-1$, and indeed, $s^\sharp_+(L) = g$ and $s^\sharp_-(L) = g-1$. 
\end{proof}

Now, we would like to show that complex curves with singularities can be put in the situation of the above lemma: 

\begin{lemma} If there is a complex curve $\Sigma$ in $B^4$ with boundary $L \subset S^3$ that has at least one singularity, then there is a complex curve $\Sigma_1$ in $B^4$ with exactly one double point, which is transverse and whose boundary is isotopic to $L$ in $S^3$.
\label{make_transverse}
\end{lemma}

\begin{proof}
We may assume that there is a singularity at $(0,0)$, by translating the picture. If the curve is cut out by $P(x,y)$, then the fact that $(0,0)$ is a singularity means that the constant and linear terms of $P$ vanish, so the lowest order terms that $P$ can have are quadratic.

Let us consider perturbing $P(x,y)$ to $P(x,y) + \epsilon xy$. For generic sufficiently small non-zero $\epsilon$, the curve cut out by this perturbation has boundary isotopic to $L$ in $S^3$, and has a transverse double point at $(0,0)$.

We must now get rid of all the other singularities. Suppose $P(x,y)$ is a polynomial that cuts out a curve with boundary $L$, with a transverse double point at $(0,0)$. Observe that for generic $\epsilon$, $P(x,y)$, $P_x(x,y) + \epsilon y$ and $P_y(x,y)+\epsilon x$ cannot all share a common factor. Thus, be Bezout's theorem $P(x,y)$ has finitely many singularities. We may rotate $\cc^2$ so that all the singularities of $P(x,y)$ except $(0,0)$ have non-zero $x$ coordinate.

We will show that for any $\epsilon_0>0$, there is $z$ with $0<|z|<\epsilon_0$ such that the only singularity of $P(x,y) + zx^2$ is at $(0,0)$, that is, we can perturb away all the other singularities.

We will show that this is possible by contradiction. Suppose that it is not true. Then, for any $z$ with $|z|<\epsilon_0$, there is some $(x,y) \neq (0,0)$ such that
\begin{align}
P(x,y) + zx^2 & = 0 \label{singularity_equations1}\\
P_x(x,y)+2zx & = 0 \label{singularity_equations2} \\
P_y(x,y) &  = 0.
\label{singularity_equations3}
\end{align}
Considering $x,y,z$ as variables, these three equations cut out a variety in $\cc^3$, such that taking the projection to the $z$ coordinate $\cc^3 \to \cc$, the image of the variety contains an open disk around $0$. 

This means that there is a component of the variety for which the image of that component alone under this projection also surjects to a (perhaps smaller) open disk around $0 \in \cc$. This component can be parametrised as $(x(z), y(z),z)$ for some smooth functions $x(z)$ and $y(z)$. 

Then, differentiating \eqref{singularity_equations1}, we get
\begin{align}
P_x(x,y) \frac{dx}{dz} + P_y(x,y) \frac{dy}{dz} + x^2 + 2zx\frac{dx}{dz} & = 0
\label{singularity2}
\end{align}
Then from \eqref{singularity_equations2}, $P_x(x,y) =-2zx$, and from \eqref{singularity_equations3}, $P_y(x,y) = 0$, so from \eqref{singularity2}, we get $x^2 = 0$ for all $z$ in the disk around $0$. This means $x(z) = 0$ for all $z$ in the disk, so $P(x,y) = 0$, $P_x(x,y) = 0$, and $P_y(x,y) = 0$. But this means in particular, that $(x(z),y(z))$ is also a singularity of $P$, so it must be $(x,y) =(0,0)$, which is a contradiction.

Thus, it cuts out the desired curve $\Sigma_1$.

\end{proof}

\begin{cor} For an algebraic knot $K$, such as a right handed torus knot, $s^\sharp(L) = 2g-1$.
\end{cor}
\begin{proof} An algebraic knot is, by definition, the link of a singular complex curve. Applying Lemma \ref{make_transverse}, we can replace it with an immersed curve with a single non-embedded point, so that that point is a transverse double point. Because the knot is connected, only one of the components of the singular complex curve in $B$ has nontrivial boundary. Thus, if the singular complex curve is not irreducible in $B$, then, viewing it as a curve in $\cc^2$, it must have a component which is entirely contained inside $B$, which is not possible. So we have a complex curve which is irreducible in $B$ that is embedded away from a single transverse double point, and applying Lemma \ref{one_double_point_lemma}, we are done.
\end{proof}

In addition, we can deduce the following.

\begin{cor} Consider the torus link $T_{md,nd}$ with $\gcd(m,n) = 1$, $m \geq n$, $d \geq 2$ and $md \geq 3$. All positive torus links with more than one component, except the Hopf link, can be written in this way. Then, $s^\sharp(T_{md,nd}) = 2g-1$, where $g$ is the genus of an embedded connected complex curve in $B^4$ whose boundary is $T_{md,nd}$.
\end{cor}

\begin{proof} By definition, $T_{md,nd}$ is the link of the singularity of $x^{md} - y^{nd} =0$ at $(0,0)$, and we can also view it as the intersection of $x^{md} - y^{nd}= 0$ with $\partial B = S^3$ for $B$ a unit ball centered at the origin. Let us consider the perturbation $x^{md} - y^{nd}-\epsilon x^2 +\delta y^2 = 0$. Observe that for sufficiently small $\epsilon$ and $\delta$, this cuts out a link isotopic to that cut out by $x^m-y^n$ on $S^3$. 

Let $P_{\epsilon,\delta}(x,y) = x^{md} - y^{nd}-\epsilon x^2 +\delta y^2$. At this point it suffices to show that for sufficiently small generic $(\epsilon,\delta)$, the surface cut out by $P_{\epsilon,\delta}$ is irreducible in $B$. 

Let us start by showing that for generic small $\epsilon$ and $\delta$, $P_{\epsilon,\delta}$ is irreducible. Since irreducibility is an open condition in the coefficients, it suffices to show that $P_{\epsilon,0}$ is irreducible for generic small $\epsilon$. Define the $(n,m)$ degree of the monomial $x^iy^j$ to be $in+mj$, and denote it $\deg_{(n,m)}(x^iy^j)$.

Suppose we $P_{\epsilon,0}$ is reducible, so that we may write it as
\[P_{\epsilon,0}(x,y) = A(x,y) B(x,y).\]
Let $A_0(x,y)$ and $B_0(x,y)$ are the top $(n,m)$ degree parts of $A(x,y)$ and $B(x,y)$. Write
\[(A_0(x,y) + A_r(x,y))(B_0(x,y) + B_r(x,y))\]
Suppose without loss of generality $\deg_{(n,m)}A_0 - \deg_{(n,m)}A_r \leq \deg_{(n,m)}B_0 - \deg_{(n,m)}B_r$. Then write $A_1(x,y)$ to be the highest $(n,m)$ degree part of $A_r$ and $B_1(x,y)$ denote the part of $B_r$ with $(n,m)$ degree $\deg_{(n,m)}B_0 - \deg_{(n,m)}A_0 + \deg_{(n,m)}A_1$, so that this is either the top $(n,m)$ degree part of $B_r$ or it is $0$.

Then, we must have that
\[A_0(x,y)B_0(x,y) = x^{md}-y^{nd}\]
\[A_1(x,y)B_0(x,y) + A_0(x,y)B_1(x,y)\]
is either $0$ or $\epsilon x^2$.

Observe that $x^m-y^n$ is irreducible (we can see this by observing that the intersection of $x^m-y^n =0$ with any ball has one component). Similarly $x^m-\omega^iy^n$ for a $\omega^i$ any $d$th root of unity. This means that
\[A_0(x,y) B_0(x,y) = (x^m-y^n)(x^m-\omega y^n) \cdots (x^m-\omega^{d-1} y^n)\]
where $\omega$ is a primitive $d$th root of unity, and $A_0(x,y)$ is the product of some of the factors $x^m-\omega^iy^n$, and $B_0(x,y)$ is the product of the rest.

In particular, this means that it is not possible that 
\[A_1(x,y)B_0(x,y) + A_0(x,y)B_1(x,y) = 0\]
because this would mean
\[A_1(x,y)B_0(x,y) =- A_0(x,y)B_1(x,y)\]
with one of the sides being non-zero, but then both sides would have to be multiples of both $A_0(x,y)$ and $B_0(x,y)$, which would mean that they would have to have $(n,m)$ degree $mnd$, but $A_1$ and $B_1$ are defined to have small enough degree that this is not possible.

Thus, we have 
\[A_1(x,y)B_0(x,y) + A_0(x,y)B_1(x,y) = \epsilon x^2.\]
This means that we cannot have that $B_1(x,y) = 0$, because if $B_1(x,y) = 0$ then $B_0(x,y)|\epsilon x^2$, but $B_0(x,y)$ is a product of $(x^m-\omega^iy^n)$ for some non-empty subset of integers $i$, and this cannot be a a factor of $x^2$. 

Observe that $A_0(x,y)$, $B_0(x,y)$, $A_1(x,y)$ and $B_1(x,y)$ are all $(n,m)$ degree homogenous in such a way that $A_1(x,y)B_0(x,y)$ and $A_0(x,y)B_1(x,y)$ are homogenous of the same degree, and this degree must be $\deg_{(n,m)} \epsilon x^2 = 2n$. But $A_0$ and $B_0$ have degrees adding up to $mnd$, so the degree of $A_1(x,y)B_0(x,y) + A_0(x,y)B_1(x,y)$ is at least $mnd/2$. Now if $md \geq 5$, we are done.

If $md = 4$ then we must have that $A_1$ and $B_1$ are non-zero of degree $0$, and also this would mean that $(A_0(x,y) + A_1(x,y))(B_0(x,y) + B_1(x,y)) = P_{\epsilon,0}(x,y)$ would have a non-trivial constant term, which is not possible.

If $md = 3$, then we must have that $d = 3$ and $m,n=1$, and $P_{\epsilon,0}(x,y) = x^3-y^3-\epsilon x^2$, and if this factors then the top total degree of one of the terms in $x$ and $y$ must be linear and therefore must be $(x-y), (x-\omega y)$ or $x-\omega^2 y$. If this factor has a constant term, then the other factor must have its lowest order term being quadratic, but its highest order term is also quadratic, so it must be homogenous of degree $2$, but then it must be a factor of $x^3-y^3$, so it cannot also be a factor of $x^2$, so this is impossible.

Thus, we have that for generic $(\epsilon,\delta)$, $P_{\epsilon,\delta}(x,y)$ is irreducible.

Now, we would like to show that for sufficiently small generic $\epsilon,\delta$, it is also irreducible in a ball. Observe that for sufficiently small $\epsilon,\delta$, we can ensure that in $B \backslash B_{1/2}$ where $B_{1/2}$ is the ball of radius $1/2$ centred at the origin, $P_{\epsilon,\delta}(x,y) = 0$ cuts out a surface sufficiently close to the one cut out by $x^m-y^n = 0$ that it also has $d$ components, one corresponding to each component of $T_{md,nd}$. Observe that outside of $B$, these $d$ irreducible components do not intersect each other.

In particular, for any small $\eta>0$, we can choose $\epsilon$ and $\delta$ such that for $|y|>\eta$, there are $md$ distinct values of $x$ such that $(x,y)$ is on the curve for each value of $y$. We can also choose $\eta$ small enough such that for $\epsilon$ and $\delta$ small enough, for $|y|<\eta$ all solutions in $x$ for each fixed $y$ are such that $(x,y)$ is inside a small ball around $(0,0)$. 

However, because the curve cut out by $P_{\epsilon,\delta}(x,y)$ is irreducible, these components are all part of the same irreducible component. They are not merging into a single component outside of $B$, so they must be merging into a single component inside of $B_{1/2}$.

For generic small $\epsilon$ and $\delta$, it is also the case that the curve cut out by $P_{\epsilon,\delta}(x,y)$ is embedded except at a single transverse double point. Applying Lemma \ref{one_double_point_lemma} gives us the desired equality.
\end{proof}

\begin{lemma} The connected sum of a quasipositive knot $K$ and a non-trivial right handed torus knot, $T_{p,q}$, bounds a complex curve with a singularity.
\end{lemma}

\begin{proof}
Let $P(x,y)$ be the polynomial and $B \subset \cc^2$ be the ball such that $\{x,y|P(x,y) = 0\} \cap \partial B$ is the knot $K$ in $S^3$, as constructed by Rudolph in \cite{rudolph}. In particular, this means that 
\[P(x,y)  = (x-r_1)(x-r_2) \cdots (x-r_n)(x-y) + \epsilon_0\]
for some real numbers $r_i$ and small $\epsilon_0$. The ball is of the form $D_x \times D_{y,\gamma}$, where $D_{y,\gamma}$ is a disk in complex plane whose boundary, $\gamma \subset \cc$ avoids the $2n$ special points $y_0$ in the complex plane that satisfy that $P(x,y) = 0$ does not have $n$ distinct solutions for $x$, and $D_x$ is some large disk in the complex plane.

Let $R$ be a positive real number such that a real four ball of radius $R$ centered at the origin in $\cc^2$ contains $B$. Then choose $a \gg b \gg R$ such that for
\[Q(x,y)  = \frac{1}{(-a)^p}(P(x,y)(x-a)^p + (y-b)^q),\] 
$Q(x,y) = 0$ cuts out the same knot in $\partial B$ as $P(x,y)$.

By abuse of notation, let us denote the resulting knot, which is isotopic to $K$, also by $K$. Let $(x_m,y_m)$ be a point on $K$ such that $x_m$ has maximal real part of $x$ coordinate. 

Observe that for a sufficiently small ball $B_\epsilon$ around $(a,b) \in \cc^2$, where $\epsilon \ll \frac{1}{a}$, $Q(x,y)=0$ cuts out the torus knot $T_{p,q}$.

\begin{claim} For generic $(a,b)$, the curve cut out by $Q(x,y) = 0$ and the curve cut out by $Q_y(x,y) = 0$ intersect at only finitely many points, where $Q_y$ is the polynomial obtained by differentiating $Q(x,y)$ with respect to $y$. 
\end{claim}
\begin{proof} By translating, it suffices to show that the curve cut out by $R(x,y) = 0$ and the curve cut out by $R_y(x,y) = 0$ intersect at only finitely many points, where 
\[R(x,y) = Q(x+a,y+b) = P(x+a,y+b)x^p + y^q = [(x-r_1+a) \cdots (x-r_n+a)(x-y+a-b) + \epsilon_0]x^p + y^q.\]
Let us write $A(x) = (x-r_1+a) \cdots (x-r_n+a)$ and $c = a-b$.

Then, we are trying to show that
\begin{align*}
(A(x) (x-y+c)+\epsilon_0)x^p+y^q &= 0\\
-A(x)x^p+qy^{q-1} = 0
\end{align*}
has finitely many solutions for $x,y \in \cc$ for generic $a$ and $c$. Observe that the second equation allows us to write 
\begin{equation}y^{q-1} = \frac{A(x)x^p}{q},\label{y_to_q-1}\end{equation}
and plugging this into the first equation, it reduces to 
\[(A(x) (x-y+c)+\epsilon_0)x^p+y\frac{A(x)x^p}{q},\]
which is now a linear equation, and can be solved in $y$ to get
\[x = 0 \text{, or } A(x) = 0 \text{, or } y = \frac{A(x) (x+c)+\epsilon_0}{A(x)\(1-\frac{1}{q}\)},\]
though actually the middle case is not possible, because if $A(x) =0$, plugging into our original equation, we would get $\epsilon_0x^p+y^q = 0$ and $y^{q-1} = 0$, which cannot simultaneously be true with $A(x) = 0$, because $0$ is not a root of $A(x)$. 

If $x=0$, then plugging into our original equations, we get $y=0$, which is one possible solution. The other case is that 
\[y = \frac{A(x) (x+c)+\epsilon_0}{A(x)\(1-\frac{1}{q}\)}.\]
This can then be plugged back into \eqref{y_to_q-1} to get 
\[\(\frac{A(x) (x+c)+\epsilon_0}{A(x)\(1-\frac{1}{q}\)}\)^{q-1} = \frac{A(x)x^p}{q},\]
which can be rewritten as
\[A(x)^q x^p(1-1/q)^{q-1} = q \(A(x)(x+c) + \epsilon_0\)^{q-1}.\]
Since, for generic $a$ and $c$, $x=0$ is not a solution to this polynomial equation in $x$, it must only have finitely many solutions $x$, and since we have written $y$ in terms of $x$ already, there are also only finitely many pairs $x,y \in \cc$ satisfying the equations, as desired. 
\end{proof}

Choose $\beta:[0,1] \to \cc$ a path in the $x$ plane from $\beta(0) = x_m$ to $\beta(1) = a$ such that it avoids projections of the intersection points of the curves $Q(x,y)=0$ and $\ddx{Q(x,y)}{y}=0$ on the $x$ plane, except at $1$, where $\beta(1) = a$, and also such that $\frac{d\Re(\beta)(t)}{dt} >0$ for all $t \in [0,1]$. 

Recall that $P(x,y)$ has degree $1$ in $y$, so the $Q(x,y)$ has degree $q$. Then for any point $x \in \beta$, there are exactly $q$ solutions $y$ to $Q(x,y) = 0$, except at $x=a$, where all $q$ solutions come together to one solution with multiplicity $q$.  For $x = x_m$ one of the $q$ solutions is in $B$.

Consider a smooth path $\beta_1:[0,1] \to \cc^2$ sitting over $\beta$, meaning that $p_x(\beta_1(t)) = \beta(t)$ for $p_x$ the projection to the $x$ coordinate, such that $\beta_1(0) = (x_m, y_m)$, $\beta_1(1) = (a,b)$ and $Q(\beta_1(t)) = 0$.

In other words, $\beta_1$ traces a path on the complex curve $Q(x,y) = 0$ so that the $x$ coordinate of the path following $\beta$.

We would now like to show that for a ball $B_1$ which is the union of $B$, $B_\epsilon$, and a sufficiently small tubular neighbourhood $N(\beta_1)$ of $\beta_1$, the intersection of $Q(x,y)$ with the boundary $S^3$ of this ball will be connected sum of $K$ with $T_{p,q}$.

For any $r$ between $\Re(x_m)$ and $a$, $\Re(x) = r$ and $Q(x,y) = 0$ are transverse at $\beta_1(t)$, where $t  \in [0,1)$ satisfies $\Re(\beta(t)) = r$. This is because $Q(x,y)$ has $\ddx{Q(x,y)(\beta_1)}{y} \neq 0$, so it is transverse to $x = \beta(t)$. This means that the tangent spaces to the surface $Q(x,y) = 0$ and the plane $\{(x,y)|x = \beta(t)\}$ span the 4 dimensional tangent space to $\cc^2$, so the surface $Q(x,y) = 0$ must also be transverse to the $3$-dimensional space $\{(x,y)|\Re(x) = \Re(\beta(t))\}$. 

This implies that for a sufficiently small tubular neighbourhood of the image of $\beta_1$, its intersection with $Q(x,y) = 0$ is a band, and its boundary are two arcs that connect two adjacent points in $L \subset \partial B$ to two adjacent points in $T_{p,q} \subset \partial B_\epsilon$, so adding in the tubular neighbourhood gives us the connected sum of the two knots.

So now we have constructed a polynomial $Q(x,y)$ with a singularity at $(a,b)$, and a ball $B$ such that the boundary of $B \cap \{(x,y)|Q(x,y)  = 0\}$ is the desired knot in $S^3 = \partial B$.
\end{proof}

As a consequence we have the following computation:

\begin{cor} For $K$ given by the connected sum of a quasi-positive knot with a non-trivial right handed torus knot, $s^\sharp(K) = 2g-1$. 
\end{cor}
\qed

In particular, we see that for the connected sum of two copies of $T_{p,q}$ for a right handed torus knot $T_{p,q}$, 
\[s^\sharp(T_{p,q} \# T_{p,q}) = 2(2g)-1 \neq 2(2g-1)=   s^\sharp(T_{p,q}),\]
so $s^\sharp$ is not additive with respect to connected sum.

However our computations leave the open the following possibilities.

\begin{question} Is $s^\sharp_+$ additive? If not, is $s^\sharp$ almost a homomorphism in the sense that $|s^\sharp(K_1 \# K_2 ) -s^\sharp(K_1) - s^\sharp(K_2)| \leq C$ for some constant $C$? 
\end{question}

\subsection{Computations of $s^\sharp_\pm(K)$ for quasi-positive $K$}

In this subsection, we will show Proposition \ref{propn_quasi-positive_inequality}, which we restate here.

\begin{propn} For a quasi-positive knot $K$ with even genus, $s^\sharp_+(K) = g$ and $g-1 \leq s^\sharp_-(K)  \leq g$. For a quasi-positive knot $K$ with odd genus, $g \leq s^\sharp_+(K) \leq g+1$ and  $s^\sharp_-(K) = g-1$.
\end{propn}

The kinds of argument that we will use will not distinguish between the case when $K$ is the unknot and the case when $K$ is a non-trivial right-handed torus knot, so in the case of even genus, it will not be possible for us to narrow down the possibilities for $s^\sharp_-(K)$ any further than this.

To do this, we will relate $s^\sharp_+$ to $s^\sharp_-$. 

We proceed with the proof of the proposition.

\begin{proof}[Proof of Proposition] Let $\Sigma$ be a punctured embedded complex curve with boundary $K$, which exists because $K$ is quasi-positive, as in \cite{rudolph}. 

Let $T = T_{4,5}$, which is an algebraic knot of even genus. Consider the cobordism $\Sigma_1$
\[U \to U \coprod U \to K \coprod T \to K \# T.\]
where the first map comes from adding a standard disk from the empty link to the unknot, the second map is $\Sigma \coprod \Sigma_T$, where $\Sigma_T$ is the complex curve with a single transverse double point, which we showed existed in a previous section, and the last piece is the usual one handle. Let $\psi:I'(K \coprod T) \to I'(K \#T)$ be the map on $I'$ induced by the last piece of $\Sigma_1$. 

The composite cobordism $\Sigma_1$ has genus $g(\Sigma) + g(\Sigma_T)$ and one positive double point, the map it induces on instantons is
\[u_+ \mapsto u_+ \otimes u_+ \mapsto \lambda^{m_+}v^K_{\pm} \otimes u_+ \mapsto \psi(\lambda^{m_+}v^K_{\pm} \otimes u_+)\]
and
\[u_- \mapsto u_- \otimes u_+ \mapsto \lambda^{m_-}v^K_{\mp} \otimes u_+ \mapsto \psi(\lambda^{m_-}v^K_{\mp} \otimes u_+).\]
However, the composite cobordism is a cobordism from the unknot to $K \# T$, which is a connected sum of a quasi-positive knot with a torus knot, and it has the same genus as the complex curve with one positive double point for $K \# T$, so it induces the same map. We have computed this map previously, and showed that $\psi^\sharp(\Sigma_1)(u^+)$ has no factors of $\lambda$, and $\psi^\sharp(\Sigma_1)(u^-)$ has exactly one factor of $\lambda$. Thus, we see that $m_+ = 0$ and $m_- \leq 1$. 

\end{proof}

\subsection{Knots that can be obtained from the unknot by switching crossings}

Finally, we compute $s^\sharp$ for knots that can be obtained from the unknot by switching only positive crossings and can also be obtained from the unknot by switching only negative crossings. 

\begin{propn} Let $D_K$ be a knot diagram for the knot $K$. Suppose that there is a subset $S_+$ of the positive crossings of $D_K$, such that switching which strand is on top for those crossings results in an unknot. Suppose further that there is a subset $S_-$ of the negative crossings such that switching which strand is on top for those crossings also yields the unknot. Then $s^\sharp_\pm(K) = 0$.
\end{propn}
\begin{proof} Let $\Sigma_-$ be the immersed cobordism from the unknot to $K$ obtained by passing the knot through itself at the negative crossings, and $\Sigma_+$ be the cobordism from $K$ to the unknot obtained by passing strands through at the positive crossings. Then the composite cobordism $\Sigma_+ \circ \Sigma_-$ from the unknot to itself has only negative double points, so $\psi^\sharp(\Sigma_+ \circ \Sigma_-)(u_+)$ and $\psi^\sharp{\Sigma_+ \circ \Sigma_-}(u_-)$ have no factors of $\lambda$, because the cobordism has no positive double points and no genus, and $s^\sharp_\pm(U) = 0$.

Thus, $\psi^\sharp(\Sigma_-)(u_+)$ and $\psi^\sharp(\Sigma_-)(u_-)$ also have no factors of $\lambda$. Because $\Sigma_-$ also has no genus and no positive double points, we may conclude that $s^\sharp_\pm(K) = 0$. 
\end{proof}

\section{Computations of $s^\sharp_I$ for some torus links}

We now show Theorem \ref{thm_s_sharp_I_links}, which we restate here.

\begin{thm} Let $m,n>1$ be relatively prime positive integers. Let $T_{md,nd}$ be the $(md,nd)$ torus link, which has $d$ components. Then 

\[s^\sharp_I(T_{md,nd}) =  \begin{cases} d\frac{(m-1)(n-1)}{2}+mn\frac{d(d-1)}{2} - n(I) & \text{if $\frac{(m-1)(n-1)}{2}$ is even}\\
d\frac{(m-1)(n-1)}{2}+d+mn\frac{d(d-1)}{2} - 3n(I) & \text{else.}\end{cases}\]
where $n(I)$ is the number of $(-1)$s in $I$.
\end{thm}
\begin{proof} The idea of the proof is to start by constructing a $d$ component immersed complex curve whose boundary is $T_{md,nd}$ such that the curve only has transverse double points and has exactly one double point on each component. Then, we will compare the maps induced by this cobordism with ones induced by cobordisms that come from perturbing away some of the double points.

Let us consider the $T_{md,nd}$ as the link of the curve cut out by
\[x^{md}-y^{nd}\]
at the point $(0,0)$. We can rewrite this as
\[ \prod_{i=0}^{d-1} (x^m-\omega^i y^n)\]
where $\omega$ is a $d$th root of unity. For a small ball $B$ around $(0,0)$, the intersection of the curve cut out by the polynomial and boundary of the ball is the torus link. Observe that each $(x^m-\omega^i y^n)$ is irreducible, and the intersection of each pair of these irreducible components is at the origin, where the intersection is not transverse. Let us start by fixing this.

For sufficiently small $a_i$, the polynomial 
\[\prod_{i=0}^{d-1} ((x-a_i)^m-\omega^i y^n )\]
also cuts out the link, $T_{md,nd}$. The irreducible components are cut out by $((x-a_i)^m-\omega^i (y)^n )$. Let us look at the intersections of these. This happens when
\begin{align*}
(x-a_i)^m &= \omega^i y^n \\
(x-a_j)^m &= \omega^j y^n \\
\end{align*}
As long as we choose the $a_i$ to be pairwise distinct, we know that at the intersection of these curves, $y \neq 0$. Then, we have 
\[(x-a_i)^m = \omega^{i-j}(x-a_j)^m,\]
so $(x-a_i) = \eta_{i-j,t} (x-a_j)$, for some $\eta_{i-j,1},\eta_{i-j,2},\ldots, \eta_{i-j,m}$ the $m$th roots of $\omega^{i-j}$. It is easy to see that for all $t$ in $\eta_t \neq 1$, so for each $t \in \{1,2\ldots, m\}$, 
\[(x-a_i) = \eta_{i-j,t} (x-a_j)\]
is a linear equation in $x$ and has exactly one solution.
\[x =\frac{ a_i -\eta_{i-j,t} a_j}{1-\eta_{i-j,t}}.\]
For each of these $m$ values of $x$, where are $n$ values of $y$ satisfying the desired equations, so there are $mn$ solutions. 
We can arrange the $a_i$ so that for different pairs $\{i,j\} \neq \{i',j'\}$, we avoid
\[\frac{ a_i - \eta_{i-j,t} a_j}{1-\eta_t} = \frac{ a_{i'} - \eta_{i'-j',s} a_{j'}}{1-\eta_s}\]
for any $s$ and $t$, because for every combination $s,t,i,j,i',j'$, the above constraint is a linear constraint on $a_i,a_j,a_i',a_j'$, and therefore are complex codimension 1, and we can easily choose the $a_i$ to avoid them.

Now, we wish to show that at these intersection points, the curves $(x-a_i)^m = \omega^i y^n$ and $(x-a_j)^m = \omega^j y^n$ intersect transversely. For this, it suffices that the gradients are linearly independent. These gradients are
\[(m(x-a_i)^{m-1}, \omega^i ny^{n-1})\]
and
\[(m(x-a_j)^{m-1}, \omega^j ny^{n-1}).\]
At the intersection points $y \neq 0$, so if these two were linearly dependent, we would have
\[(x-a_i)^{m-1} = \omega^{i-j} (x-a_j)^{m-1},\]
But this is not possible, because at the intersection points, we know that $\frac{x-a_i}{x-a_j}$ is an $m$th root of $\omega^{i-j}$, so it cannot also be an $(m-1)$th root of $\omega^{i-j}$, since $\omega^{i-j} \neq 1$, so we cannot have $\eta^m = \omega^{i-j}$ and $\eta^{m-1} = \omega^{i-j}$. Thus we have that the intersection points of the different components are transverse.

To summarise, we have shown that we can choose arbitrarily small $a_i$ so that the polynomial 
\[\prod_{i=0}^{d-1} ((x-a_i)^m-\omega^i y^n )\]
cuts out that $T_{md,nd}$ torus link, and the irreducible components of if, which are cut out by $((x-a_i)^m-\omega^i y^n )$, intersect transversely, at $mn$ points for each pair $(i,j)$, and such that in total there are $mn\frac{d(d-1)}{2}$ of these intersection points, and all of them are transverse singularities of the polynomial.

Each component, however, still has a singularity which is not a transverse double point. Let us fix this by choosing small $\epsilon_i$ and replacing our polynomial with
\[\prod_{i=0}^{d-1} ((x-a_i)^m-\omega^i y^n-\epsilon_i(x-a_i)y ).\]
For $\epsilon_i$ small enough, it is still the case that the $((x-a_i)^m-\omega^i y^n-\epsilon_i(x-a_i)y )$ are irreducible and the corresponding curves intersect each other transversely at distinct points, so that there are a total of $mn\frac{d(d-1)}{2}$ intersections. However, for sufficiently small $\epsilon_i>0$, we can arrange that $((x-a_i)^m-\omega^i y^n-\epsilon_i(x-a_i)y )$ has only transverse double points as singularities. To see this, it suffices to show that for sufficiently small $\epsilon>0$, all the singularities of
\[x^m-y^n-\epsilon xy=0\]
are transverse double points.
A singularity is a point that also satisfies $\ddx{(x^m-y^n-\epsilon xy)}{x}=0$ and $\ddx{(x^m-y^n-\epsilon xy)}{y}=0$, so
\begin{align*}
x^m-y^n-\epsilon xy &= 0\\
mx^{m-1}-\epsilon y & = 0\\
-ny^{n-1}-\epsilon x & = 0\\
\end{align*}
We can write the latter two equations as 
\begin{align*}
x^{m} & = \frac{\epsilon x y}{m}\\
y^{n}& = \frac{-\epsilon x y }{n}\\
\end{align*}
Plugging these into the original equation, we get
\[\epsilon x y \(\frac{1}{m}+\frac{1}{n} - 1\) = 0\]
and since $m$ and $n$ are relatively prime, we get that this means $ x= 0$ or $y = 0$, and to satisfy our equations, we would have to have $x=y=0$. Thus, the only singularity is at $x=y=0$, and this is a transverse double point.

We have now produced the polynomial,
\[P(x,y)  = \prod_{i=0}^{d-1} ((x-a_i)^m-\omega^i y^n-\epsilon_i(x-a_i)y )\]
cutting out the torus link $T_{md,nd}$ such that the $((x-a_i)^m-\omega^i y^n-\epsilon_i(x-a_i)y )$ are irreducible. Overall, the surface cut out is immersed. There are $mn\frac{d(d-1)}{2}+d$ double points, where the $d$ extra double points are $(a_i,0)$. Additionally, notice that each $((x-a_i)^m-\omega^i y^n-\epsilon_i(x-a_i)y )$ cuts out the $m,n$ torus knot, and it has one double point, and as we have seen above in our examination of immersed complex surfaces that cut out the torus knot in the sphere, it has genus $\frac{(m-1)(n-1)}{2}-1$. Thus, overall we have an immersed complex curve with $d$ components, $mn \frac{d(d-1)}{2}+d$ double points, one on each individual component, and $mn \frac{d(d-1)}{2}$ between components, and genus $\frac{(m-1)(n-1)}{2}-1$ on each component.

We will now compute the number of factors of $\lambda$ the generator $u_I$ picks up under the map induced by this curve, with a ball removed from each component, as a cobordism between the $d$ component unlink to $T_{md,nd}$. 

We now proceed as we did in the previous section: In addition to the curve cut out by $P(x,y)$ consider the curve cut out by 
\[P_1(x,y) = ((x-a_i)^m-y^n-\epsilon) \prod_{i=1}^{d-1} ((x-a_i)^m-\omega^i y^n-\epsilon_i(x-a_i)y ).\]
This is also a complex curve, and we can still choose it so that the components intersect transversely at the same number of points as before, but now $((x-a_i)^m-y^n-\epsilon)$ has no double points and has genus $\frac{(m-1)(n-1)}{2}$. 

Let $\Sigma$ and $\Sigma_1$ denote the complex curves cut out by $P$ and $P_1$ respectively, with a disk removed from each component.

Now, consider the cobordism $T_1:U_d \to U_d$ given by a twice punctured $\tt^2$ on the first component, and a cylinder for each other component, and $D_1: U_d \to U_d$, given by a cylinder with a single positive double point on the first component, and a cylinder for each other component.

Then, by the same argument as in Lemma \ref{one_double_point_lemma}, using the maps induced by $\Sigma_1 \circ D_1$ and $\Sigma \circ T_1$ are the same, and from this we can see that for a generator $u_I = u_{i_1} \otimes \cdots \otimes u_{i_d}$, if $u_{i_1} = u_-$, then $\psi^\sharp(\Sigma_1)(u_I)$ and $\psi^\sharp(\Sigma)(u_I)$ are both multiples of $\lambda$. We can apply the same argument for each of the components, so we get that $\psi^\sharp(\Sigma)u_I$ is a multiple of $\lambda$ if any of the $u_{i_d}$ are $u_-$. However, by the same argument as in Lemma \ref{one_double_point_lemma}, we also have that there is some $u_I$ which maps to a non-multiple of $\lambda$, since the $u_I$ generate $I'(U_d)$. Thus, $m_I(\Sigma) = 0$ for $I$ corresponding to $u_+ \otimes \cdots \otimes u_+$. 

From this, and by comparing $\Sigma_1 \circ D_1$ and $\Sigma \circ T_1$, as in Lemma \ref{one_double_point_lemma}, we can deduce that $m_I(\Sigma)$ is the number of instances of $u_-$ that appears in $u_I$. The value of $s^\sharp_I(T_{md,nd})$ follows.

\end{proof}

\bibliographystyle{abbrv}

\bibliography{s_sharp_links}

\end{document}